\documentclass[12pt]{article}
\usepackage{amsmath,amssymb,amsbsy,amsfonts,amsthm,latexsym,amsopn,amstext,
                            amsxtra,euscript,amscd}

\newfont{\teneufm}{eufm10}
\newfont{\seveneufm}{eufm7}
\newfont{\fiveeufm}{eufm5}
\newfam\eufmfam
                 \textfont\eufmfam=\teneufm \scriptfont\eufmfam=\seveneufm
                 \scriptscriptfont\eufmfam=\fiveeufm
\def\bbbc{{\mathchoice {\setbox0=\hbox{$\displaystyle\rm C$}\hbox{\hbox
to0pt{\kern0.4\wd0\vrule height0.9\ht0\hss}\box0}}
{\setbox0=\hbox{$\textstyle\rm C$}\hbox{\hbox
to0pt{\kern0.4\wd0\vrule height0.9\ht0\hss}\box0}}
{\setbox0=\hbox{$\scriptstyle\rm C$}\hbox{\hbox
to0pt{\kern0.4\wd0\vrule height0.9\ht0\hss}\box0}}
{\setbox0=\hbox{$\scriptscriptstyle\rm C$}\hbox{\hbox
to0pt{\kern0.4\wd0\vrule height0.9\ht0\hss}\box0}}}}
\def\bbbq{{\mathchoice {\setbox0=\hbox{$\displaystyle\rm
Q$}\hbox{\raise
0.15\ht0\hbox to0pt{\kern0.4\wd0\vrule height0.8\ht0\hss}\box0}}
{\setbox0=\hbox{$\textstyle\rm Q$}\hbox{\raise
0.15\ht0\hbox to0pt{\kern0.4\wd0\vrule height0.8\ht0\hss}\box0}}
{\setbox0=\hbox{$\scriptstyle\rm Q$}\hbox{\raise
0.15\ht0\hbox to0pt{\kern0.4\wd0\vrule height0.7\ht0\hss}\box0}}
{\setbox0=\hbox{$\scriptscriptstyle\rm Q$}\hbox{\raise
0.15\ht0\hbox to0pt{\kern0.4\wd0\vrule height0.7\ht0\hss}\box0}}}}
\def\bbbt{{\mathchoice {\setbox0=\hbox{$\displaystyle\rm
T$}\hbox{\hbox to0pt{\kern0.3\wd0\vrule height0.9\ht0\hss}\box0}}
{\setbox0=\hbox{$\textstyle\rm T$}\hbox{\hbox
to0pt{\kern0.3\wd0\vrule height0.9\ht0\hss}\box0}}
{\setbox0=\hbox{$\scriptstyle\rm T$}\hbox{\hbox
to0pt{\kern0.3\wd0\vrule height0.9\ht0\hss}\box0}}
{\setbox0=\hbox{$\scriptscriptstyle\rm T$}\hbox{\hbox
to0pt{\kern0.3\wd0\vrule height0.9\ht0\hss}\box0}}}}
\def\bbbs{{\mathchoice
{\setbox0=\hbox{$\displaystyle     \rm S$}\hbox{\raise0.5\ht0\hbox
to0pt{\kern0.35\wd0\vrule height0.45\ht0\hss}\hbox
to0pt{\kern0.55\wd0\vrule height0.5\ht0\hss}\box0}}
{\setbox0=\hbox{$\textstyle        \rm S$}\hbox{\raise0.5\ht0\hbox
to0pt{\kern0.35\wd0\vrule height0.45\ht0\hss}\hbox
to0pt{\kern0.55\wd0\vrule height0.5\ht0\hss}\box0}}
{\setbox0=\hbox{$\scriptstyle      \rm S$}\hbox{\raise0.5\ht0\hbox
to0pt{\kern0.35\wd0\vrule height0.45\ht0\hss}\raise0.05\ht0\hbox
to0pt{\kern0.5\wd0\vrule height0.45\ht0\hss}\box0}}
{\setbox0=\hbox{$\scriptscriptstyle\rm S$}\hbox{\raise0.5\ht0\hbox
to0pt{\kern0.4\wd0\vrule height0.45\ht0\hss}\raise0.05\ht0\hbox
to0pt{\kern0.55\wd0\vrule height0.45\ht0\hss}\box0}}}}
\def\bbbz{{\mathchoice {\hbox{$\sf\textstyle Z\kern-0.4em Z$}}
{\hbox{$\sf\textstyle Z\kern-0.4em Z$}}
{\hbox{$\sf\scriptstyle Z\kern-0.3em Z$}}
{\hbox{$\sf\scriptscriptstyle Z\kern-0.2em Z$}}}}

\newtheorem{thm}{Theorem}
\newtheorem{lemma}[thm]{Lemma}
\def\cB{{\mathcal B}}
\def\cC{{\mathcal C}}

\def\cX{{\mathcal X}}

\def\({\left(}
\def\){\right)}
\def\[{\left[}
\def\]{\right]}
\def\<{\langle}
\def\>{\rangle}

\def\F{\mathbb{F}}

\def\Z{\mathbb{Z}}
\def\T{\mathbb{T}}

\def\R{\mathbb{R}}
\def\C{\mathbb{C}}

\def\Fp{\F_p}

\def\Tm{\T_m}
\def\Tn{\T_n}
\def\Ts{\T_s}

\def\vec#1{\mathbf{#1}}
\def\dist{\mathrm{dist}}

\def\Fmn{\mathfrak{F}_{m,n}}

\def\Gmnp{\mathfrak{G}_{m,n,p}}

\def\fC{\mathfrak{C}}

\def\mand{\qquad\mbox{and}\qquad}

\begin{document}

\title{\bf On the Distribution of Values and Zeros of Polynomial
Systems over Arbitrary Sets}

\author{
{\sc   Bryce Kerr} \\
{Department of Computing, Macquarie University} \\
{Sydney, NSW 2109, Australia} \\
{\tt  bryce.kerr@mq.edu.au}
\and
{\sc   Igor E.~Shparlinski}\thanks{This work was supported 
in part by ARC
Grant~DP1092835} \\
{Department of Computing, Macquarie University} \\
{Sydney, NSW 2109, Australia} \\
{\tt  igor.shparlinski@mq.edu.au}}

\date{}

\maketitle

\begin{abstract}  Let $G_1,\ldots, G_n \in \Fp[X_1,\ldots,X_m]$  
be $n$ polynomials in $m$ variables over the finite field $\Fp$ of $p$ elements. 
A result of {\'E}.~Fouvry and N.~M.~Katz shows that under some natural condition, 
for any fixed $\varepsilon$ 
and sufficiently large prime $p$ the vectors of fractional parts
$$
\(\left\{\frac{G_1(\vec{x})}{p}\right\}, \ldots,\left\{\frac{G_n(\vec{x})}{p}\right\} \), 
\qquad \vec{x} \in \Gamma,
$$ 
are uniformly distributed in the unit cube 
$[0,1]^n$ for any cube $\Gamma \in [0, p-1]^m$ with the side length 
$h \ge p^{1/2} (\log p)^{1 + \varepsilon}$. Here we use this result to show 
the above vectors remain uniformly distributed, 
when $\vec{x}$ runs through a rather general set. We also obtain 
new results about the distribution of solutions to system of 
polynomial congruences. 
\end{abstract}


\section{Introduction}

Let $p$ be a prime and let $\Fp$ be the finite field of $p$ elements, 
which we assume to be represented by the set $\{0, 1, \ldots, p-1\}$.

Given $n$ polynomials $G_j(X_1,\ldots,X_m) \in \Fp[X_1,\ldots,X_m]$, $j=1, \ldots, n$,
in $m$ variables
with integer coefficients, 
we consider the following points formed by fractional parts:
\begin{equation}
\label{eq:points}
\(\left\{\frac{G_1(\vec{x})}{p}\right\}, \ldots,\left\{\frac{G_n(\vec{x})}{p}\right\} \), 
\qquad  \vec{x} = (x_1, \ldots, x_m) \in \Fp^m.
\end{equation}

We say that the polynomials $G_1, \ldots, G_n$ are {\it degree $2$ independent over $\Fp$\/} if any 
non-trivial linear combinations $a_1G_1+\ldots + a_nG_n$ is a polynomial 
of degree at least $2$ over $\Fp$.
 
Let $\Gmnp$  denote the family
of  polynomial systems $\{G_1, \ldots, G_n\}$ of $n$ polynomials in  $m$ variables
that are degree $2$ independent over $\Fp$.

Fouvry and Katz~\cite{FoKa} have shown that for any 
 $\{G_1, \ldots, G_n\}\in \Gmnp$, the points~\eqref{eq:points} are 
uniformly distributed in the unit cube 
$[0,1]^n$, where $\vec{x}$ runs through the integral points
in  any cube $\Gamma \in [0, p-1]^m$ with side length 
$h \ge p^{1/2} (\log p)^{1 + \varepsilon}$.
Here we use several of the results from~\cite{FoKa} combined with 
some ideas of Schmidt~\cite{Schm} to obtain a similar uniformity of 
distribution result when $\vec{x}$ runs through a set from a very 
general family. 
For example, this holds for $\vec{x}$ that belong to 
the dilate  $p\Omega$ of a convex set $\Omega \in [0, 1]^m$ of 
Lebesgue measure at least $p^{-1/2 + \varepsilon}$ for any fixed $\varepsilon > 0$
and sufficiently large $p$.
We note that standard way of moving from boxes to arbitrary 
convex sets, via the isotropic discrepancy, see~\cite[Theorem~2]{Schm},
leads to a much weaker result which is nontrivial only for 
sets $\Omega \in [0, 1]^m$ of 
Lebesgue measure at least $p^{-1/2m + \varepsilon}$.

As in~\cite{Shp}, it is crucial for our approach that the error term 
in the aforementioned asymptotic formula of~\cite{FoKa}
depends on the size of the cube  $\Gamma \in [0, p-1]^m$
and decreases rapidly together with the size of $\Gamma$. 
We note that a similar idea has also recently been used in~\cite{Ker}
in combination with a new upper bound on the number of zeros 
of multivariate polynomial congruences in small cubes.  

Furthermore, given $n$ polynomials $F_j(X_1,\ldots,X_m) \in \Z[X_1,\ldots,X_m]$, 
$j=1, \ldots, n$, 
we consider the distribution of points in the set $\cX_{p}$,
of solutions
$\vec{x} = (x_1, \ldots, x_m) \in \Fp^m$ to the system of congruences
\begin{equation}
\label{eq:syst}
F_j(\vec{x}) \equiv 0 \pmod p, \qquad j =1, \ldots, n.
\end{equation}

Let  $\Fmn$ denote the family
of  polynomial systems $\{F_1, \ldots, F_n\}$ of $n$ polynomials in  $m\ge n+1$ variables
with integer coefficients, such that the solution set of 
the  system of equations (over $\C$)
$$
F_j(\vec{x}) = 0, \qquad j =1, \ldots, n,
$$ 
has at least one   
absolutely irreducible  component of dimension $m-n$ and no absolutely irreducible  component 
is contained in a hyperplane.
For sufficiently large $p$ all absolutely irreducible  components
remain of the same dimension and are absolutely irreducible modulo $p$,
so by the Lang-Weil theorem~\cite{LaWe}
we have
\begin{equation}
\label{LaWe}
 \# \cX_p = \nu p^{m-n} + O\(p^{m-n-1/2}\),
\end{equation}
where $\nu$ is the number of absolutely irreducible components of $\cX_p$
of dimension $m-n$.  
It is shown in~\cite{Shp}, that for  a rather general class of sets $\Omega$, 
including all convex sets, we have
\begin{equation}
\label{eq:Omega} 
 T_p(\Omega) =  \# \cX_p \(\mu(\Omega) +O\(p^{-1/2(n+1)}\log p\)\)
\end{equation}
with
$$T_p(\Omega)= \# (\cX_p \cap \Omega).$$The asymptotic formula~\eqref{eq:Omega} is  based on a 
combination of a result of Fouvry~\cite{Fouv} and
Schmidt~\cite{Schm}. 

Here we show that for a more restricted class of sets, which includes such natural
sets as $m$-dimensional balls,  one can improve~\eqref{eq:Omega} and 
obtain an asymptotic formula which is nontrivial provided that 
$$
\mu(\Omega) \ge p^{-1/2+\varepsilon}
$$ 
for any fixed $\varepsilon> 0$ and a sufficiently large $p$,
while~\eqref{eq:Omega}  is nontrivial only under the 
condition $\mu(\Omega) \ge p^{-1/2(n+1)+\varepsilon}$ (but applies to a
wider class of sets).

\section{Well and Very Well Shaped Sets}

Let $\Ts= (\R/\Z)^s$ be the $s$-dimensional unit torus.

We define the  distance between a vector $\vec{u} \in \Tm$
and a set $\Omega\subseteq\Tm $  by
$$
\dist(\vec{u},\Omega) = \inf_{\vec{w} \in\Omega}
\|\vec{u} - \vec{w}\|,
$$
where  $\|\vec{v}\|$ denotes the Euclidean norm of $\vec{v}$. Given
$\varepsilon >0$ and a set  $\Omega \subseteq\Tm $, we define
the  sets
$$
\Omega_\varepsilon^{+} = \left\{ \vec{u} \in \Tm \backslash
\Omega \ : \ \dist(\vec{u},\Omega) < \varepsilon \right\}
$$
and
$$
\Omega_\varepsilon^{-} = \left\{ \vec{u} \in \Omega \ : \
\dist(\vec{u},\Tm \backslash \Omega )  < \varepsilon  \right\} .
$$

We say that a set $\Omega$ is {\it well shaped\/} if 
\begin{equation}
\label{eq:Blowup}
\mu\(\Omega_\varepsilon^{\pm }\)  \le C \varepsilon, 
\end{equation}
for some constant $C$, where $\mu$ is the Lebesgue measure on $\Tm$.

It is known that any convex set is well shaped, 
see~\cite[Lemma~1]{Schm}.

Finally,  a very general result of Weyl~\cite[Equation~(2)]{Weyl} 
(taken with $n = m$ and $\nu = n-1$),
that actually expresses $\mu\(\Omega_\varepsilon^{\pm }\)$ 
via a finite sum of powers $\varepsilon^i$, $i = 1, \ldots, m$
in the case when the boundary of $\Omega$ is manifold of dimension $n-1$.
Examining the constants in this expansion we see that any such set with 
a bounded surface size is well shaped.

Furthermore, we say that a set $\Omega \subseteq \Tm$ is \emph{very well shaped} if for every $\varepsilon>0$ the measures of the sets 
$\Omega_{\varepsilon}^{\pm}$
exist and satisfy
\begin{equation}
\label{very well shaped}
\mu(\Omega_{\varepsilon}^{\pm})\leq C\(\mu(\Omega)^{1-1/m}\varepsilon+\varepsilon^{m}\)
\end{equation}
for some  $C>0$, the most natural example of a very well shaped set being a Euclidian ball.

We recall that the notation $A(t) \ll B(t)$  is equivalent to 
$A(t) =O(B(t))$, which means that there exists some absolute constant, $\alpha$,  such that
$|A(t)|\leq \alpha B(t)$ for all values of $t$ within a certain range.
Throughout the paper, the implied constants in symbols `$O$' and `$\ll$'
may depend on the constant $C$ in~\eqref{eq:Blowup} and~\eqref{very well shaped} and 
it may also depend on the polynomial system  $\{F_1,\dots, F_n\}\in \Fmn$  
(but does not depend on the polynomial system  $\{G_1,\dots, G_n\} \in \Gmnp$).

\section{Discrepancy}

Given a sequence $\Xi$ of $N$ points 
\begin{equation}
  \label{eq:GenSeq}
  \Xi = \{(\xi_{k,1}, \ldots, \xi_{k,n})\}_{k=1}^{N},
  \end{equation}
in  $\Tn$, 
we define its {\it discrepancy\/}
as
$$
\varDelta(\Xi) = \sup_{\varPi \subseteq \Tn}\left| \frac{ \#A(\Xi,\varPi)}{N} - \lambda(\varPi)\right|,
$$
 where  $A(\Xi,\varPi)$ is the number of  
$k\le N$ such that $(\xi_{k,1}, \ldots, \xi_{k,n}) \in \varPi$,  $\lambda$ is the Lebesgue measure on $\Tn$
and  the supremum is taken over all boxes 
\begin{equation}
  \label{eq:Box}
  \varPi = [\alpha_1,
\beta_1) \times \ldots \times [\alpha_n, \beta_n) \subseteq \Tn,
  \end{equation}
see~\cite{DrTi,KuNi}.  

We also define the discrepancy of an empty sequence as 1.

\section{Main Results}

For a set $\Omega \subseteq \Tm$ 
let $D(\Omega)$ be the discrepancy of  the points
$$
\(\left\{\frac{G_1(\vec{x})}{p}\right\}, \ldots,\left\{\frac{G_n(\vec{x})}{p}\right\} \), 
\qquad   \vec{x}\in p\Omega.
$$

\begin{thm}
\label{thm:Omega} 
For any   polynomial system  $\{G_1,\dots, G_n\} \in \Gmnp$ 
and any well shaped set $\Omega\in \Tm$,  we have  
$$ 
D(\Omega) \ll \mu(\Omega)^{-1} p^{-1/2}(\log p)^{n+2}.
$$
\end{thm} 

We can get a sharper error term for the case of very well shaped sets.
\begin{thm}
\label{thm:Omega vws}
For any   polynomial system  $\{G_1,\dots, G_n\} \in \Gmnp$  
and any very well shaped set $\Omega\in \Tm$,  we have    
$$ 
D(\Omega) \ll \mu(\Omega)^{-1/m}p^{-1/2}(\log p)^{n+2}. 
$$
\end{thm} 

We prove the following
\begin{thm}
\label{system}
For any polynomial system $\{F_1,\dots,F_n\}\in \Fmn$  
and any  very well shaped set $\Omega\in \Tm$, we have
$$
T_p(\Omega)= \# \cX_p \( \mu(\Omega)+O 
\(\mu(\Omega)^{1-1/m}p^{-1/2(n+1)}\log{p}+p^{-1/2}(\log{p})^{n+2}\)\).
$$
\end{thm}

\section{Exponential Sum and Congruences}
\label{sec:Exp}

Typically the bounds on the discrepancy of a 
sequence  are derived from bounds of exponential sums
with elements of this sequence. 
The relation is made explicit in 
the celebrated {\it Koksma--Sz\"usz
  inequality\/}, see~\cite[Theorem~1.21]{DrTi},
which we  present in the following form.

\begin{lemma}
  \label{lem:Kok-Szu}
   Suppose that for the sequqnce of points~\eqref{eq:GenSeq} for some integer $L \ge 1$ and the real number $S$ we have
  \begin{equation*}
    \left | \sum_{k=1}^{N} \exp \( 2 \pi i\sum_{j=1}^{n}a_j\xi_{k,j} \) \right |\le S, 
      \end{equation*}
 for all integers $-L\le a_j\le L$, $j=1, \ldots, n$,  not all equal to zero. Then, 
  \begin{equation*}
      D(\Gamma) \ll \frac{1}{L}+\frac{(\log L)^n}{N} S,
  \end{equation*}
  where the implied constant depends only on $n$.
\end{lemma} 

To use Lemma~\ref{lem:Kok-Szu} we need the following bound of 
 Fouvry and Katz~\cite[Equation~(10.6)]{FoKa} 

\begin{lemma}
  \label{lem:ExpSums}
For any polynomial system $\{G_1,\dots, G_n\} \in \Gmnp$ 
and  arbitrary integers $u$ and $w$ 
with $1 \le w< p$, uniformly over all  non-zero modulo $p$ integer vectors  
$(a_1, \ldots, a_n)$ we have  
$$ 
\sum_{x_1, \ldots, x_m=u}^{u+w} 
\exp\( \frac{2 \pi i}{p} \sum_{j=1}^n a_j G_j(x_1,\ldots,x_m)\) \ll p^{1/2}w^{m-1} \log p.
$$
\end{lemma} 

\begin{proof}  The bound in~\cite[Equation~(10.6)]{FoKa}, that gives the desired result
for $u=0$ is   uniform in polynomials $G_1, \ldots, G_m$. It now remains to 
notice that the property of being degree $2$ independent is preserved under the 
change of variables $X_j \to X_j + u$, $j=1, \ldots, m$.
\end{proof}

The proof of Theorem~\ref{system} is based on the following bound for $T_p(\cC)$ for a cube $\cC$ 
which is essentially a result of Fouvry~\cite{Fouv}
\begin{lemma}
\label{fouv} For any  polynomial  system   $\{F_1, \ldots, F_n\} \in \Fmn$   
and  any cubic box
$$\cC=\left[\gamma_1+\frac{u_1}{k},\gamma_1+\frac{u_1+1}{k}\right]\times \dots \times  \left[\gamma_{m}+\frac{u_{m}}{k},\gamma_{m}+\frac{u_{m}+1}{k}\right]  \subseteq \R^m,
$$
where $u_1,\dots, u_m \in \Z$, of side length $1/k$, we have
\begin{equation*}
\begin{split}
T_p(\cC) =\#\cX_p & \(\frac{1}{k}\)^m \\
& +
O\(p^{(m-n)/2}(\log p)^m+ \(\frac{1}{k}\)^{m -n-1}p^{m-n-1/2}(\log p)^{n+1}\).
\end{split} 
\end{equation*}
\end{lemma}

\section{Proof of Theorem~\ref{thm:Omega}}

For a set $\Omega \subseteq \Tm$ and a  box  
$\varPi   \subseteq \Tn$ of the form~\eqref{eq:Box}
let $N(\Omega;\varPi)$ be the number of integer vectors $\vec{x}\in p\Omega$
for which the points~\eqref{eq:points} belong to $\varPi$. 

In particular,  let $N(\Omega)=N(\Omega;\Tm)$ be the number of 
integer vectors $\vec{x}\in p\Omega$.
A simple geometric argument shows that if $\Omega = \Gamma \subseteq \Tm$ is a cube
then 
\begin{equation}
\label{eq:N cube 1}
N(\Gamma) = \mu(\Gamma)p^{m}  +  O\(p^{m-1}\mu(\Gamma)^{(m-1)/m}\).
\end{equation}

We start with deriving a lower bound on $N(\Omega;\varPi)$. 

We now recall some constructions and arguments from the proof of~\cite[Theorem~2]{Schm}. 
Pick a point $\boldsymbol\gamma  = (\gamma_1, \ldots, \gamma_m) \in \Tm$ such that all its 
coordinates are irrational. 
For positive $k$, let $\fC(k)$ be the set of cubes of the form
$$
\left[\gamma_1 + \frac{u_1}{k}, \gamma_1 + \frac{u_1+1}{k}\right]\times \ldots  \times \left[\gamma_m + \frac{u_m}{k}, \gamma_m + \frac{u_m+1}{k}\right]
\subseteq \R^m,
$$
where $u_1,\dots, u_m \in \Z$. Note that the above irrationality condition guarantees that the 
points $p^{-1} \vec{x}$ with $\vec{x} \in \Z^m$ 
all belong to the interior of the cubes from $\fC(k)$. 

Furthermore, let $\cC(k)$ be the set of cubes from $\fC(k)$ that are contained inside
of $\Omega$. By~\cite[Equation~(9)]{Schm},  for
any  well shaped set $\Omega\in \Tm$,  we have
\begin{equation}
\label{eq:Card Ck}
\#\cC(k) = k^m\mu(\Omega) + O(k^{m-1}).
\end{equation}
Let $\cB_1 = \cC(2)$ and for $i =2,3, \ldots$, 
let $\cB_i$ be the set of cubes $\Gamma \in \cC(2^i)$ that are not contained in any cube from $\cC(2^{i-1})$.
Clearly
$$
2^{-im} \# \cB_i  + 2^{-(i-1)m} \#\cC(2^{i-1})  \le \mu(\Omega), \qquad i=2,3, \ldots.
$$
We now see from~\eqref{eq:Card Ck} that 
\begin{equation}
\label{eq:Card}
\# \cB_i \ll 2^{i(m-1)}
\end{equation}
and also for any integer $M\ge 1$,
$$
\Omega  \setminus \Omega_\varepsilon^{-}  \subseteq \bigcup_{i=1}^M  \bigcup_{\Gamma \in \cB_i}\Gamma  \subseteq \Omega
$$
with $\varepsilon =  m^{1/2}2^{-M}$.  
Since $\Omega$ is well shaped, we obtain
\begin{equation}
\label{eq:Aprox}
 \mu\(\bigcup_{i=1}^M  \bigcup_{\Gamma \in \cB_i}\Gamma\) =  \mu\(\Omega\) + O(2^{-M}).
\end{equation}

Using Lemma~\ref{lem:Kok-Szu} (taken with $L=(p-1)/2$) and recalling 
the bound of Lemma~\ref{lem:ExpSums} we see that the 
discrepancy $D(\Gamma)$ of the points~\eqref{eq:points}
with $\vec{x} \in p\Gamma$, for a cube $\Gamma$ satisfies 
$$ 
D(\Gamma)  \ll \frac{p^{1/2} \(p\mu(\Gamma)^{1/m}\)^{m-1}}{\mu(\Gamma) p^m}(\log p)^{n+1}
= p^{-1/2} \mu(\Gamma)^{-1/m} (\log p)^{n+1}.
$$
Therefore, using~\eqref{eq:N cube 1}, we derive
\begin{equation} 
\begin{split}
\label{eq:N cube 2}
N(\Gamma;\varPi)  & = \lambda(\varPi) N(\Gamma) + O( N(\Gamma)  D(\Gamma) )\\
& = \lambda(\varPi)  \mu(\Gamma) p^m +  O\(p^{m-1/2} \mu(\Gamma)^{(m-1)/m}(\log p)^{n+1}\).
\end{split} 
\end{equation} 
Hence
\begin{equation}
\label{eq:Prelim}
N(\Omega;\varPi)\ge \sum_{i=1}^M \sum_{\Gamma \in \cB_i}N(\Gamma;\varPi) 
= \lambda(\varPi) p^m \sum_{i=1}^M \sum_{\Gamma \in \cB_i} \mu(\Gamma) + O(R),
\end{equation}
where
$$R
= p^{m-1/2} (\log p)^{n+1}\sum_{i=1}^M  \# \cB_i 2^{-i(m-1)} .
$$
We see from~\eqref{eq:Aprox}  that 
\begin{equation}
\label{eq:Main T}
\sum_{i=1}^M \sum_{\Gamma \in \cB_i} \mu(\Gamma)=    \mu\(\bigcup_{i=1}^M  \bigcup_{\Gamma \in \cB_i}\Gamma\) = 
\mu\(\Omega\) +  O(2^{-M}).
\end{equation}
 
Furthermore, using~\eqref{eq:Card}, we derive
\begin{equation}
\label{eq:Error T}
\begin{split}
R  \ll M p^{m-1/2} (\log p)^{n+1}.
\end{split} 
\end{equation}
We now choose $M$ to satisfy 
$$
2^{M }\le  p^{1/2}  < 2^{(M+1)}.
$$
Now, substituting~\eqref{eq:Main T} and~\eqref{eq:Error T} 
in~\eqref{eq:Prelim} with the above  choice of $M$, we obtain 
 \begin{equation}
\label{eq:LB}
N(\Omega;\varPi)\ge \lambda(\varPi)   \mu(\Omega) p^m+  O\(p^{m-1/2}(\log p)^{n+2}\).
\end{equation}

Since the complementary set $\overline \Omega = \Tm\setminus \Omega$ 
is also well shaped,   we also have
 \begin{equation}
\label{eq:UB}
N(\overline \Omega;\varPi)\le \lambda(\varPi)  
\mu(\overline \Omega) p^m +   O\(p^{m-1/2}(\log p)^{n+2}\).
\end{equation}
Note that by~\eqref{eq:N cube 2} we have
$$
N(\Tm;\varPi) = \lambda(\varPi)  p^m+  O\(p^{m-1/2}(\log p)^{n+1}\).
$$
Now, since 
$$
N(\overline \Omega;\varPi) = N(\Tm;\varPi)-N(\varPi)
\mand \mu(\overline \Omega)  = 1 - \mu( \Omega) ,
$$
we now see that~\eqref{eq:UB} implies that upper bound
$$
N(\Omega;\varPi) \le \lambda(\varPi)   \mu(\Omega) p^m+  O\(p^{m-1/2}(\log p)^{n+2}\)
$$ 
together with~\eqref{eq:LB} leads to the desired asymptotic formula
$$
N(\Omega;\varPi) = \lambda(\varPi)   \mu(\Omega) p^m+  O\(p^{m-1/2}(\log p)^{n+2}\).
$$ 

Since $D(\Omega) \le 1$, we can   assume that
$$
\mu(\Omega) \ge c_0 p^{-1/2}(\log p)^{n+2}.
$$
for a sufficiently large constant $c_0 > 0$
as otherwise the result is trivial. 
Thus
$$
\frac{N(\Omega;\varPi)}{N(\Omega)}  = \lambda(\varPi)  +  O\( \mu(\Omega)^{-1} p^{-1/2}(\log p)^{n+2}\)
$$
which concludes the proof. 

\section{Proof of Theorem~\ref{thm:Omega vws}}

 If $\Omega$ is very well shaped we may use the same method as the proof of Theorem~\ref{thm:Omega} to replace the bounds ~\eqref{eq:Card} and ~\eqref{eq:Aprox} with
\begin{equation}
\label{eq:Card vws}
\#\cB_i \ll 1+ \mu(\Omega)^{(m-1)/m}2^{i(m-1)}
\end{equation}
and
\begin{equation}
\label{eq:Aprox vws}
 \mu\(\bigcup_{i=1}^M  \bigcup_{\Gamma \in \cB_i}\Gamma\) =  \mu\(\Omega\) + O(\mu(\Omega)^{(m-1)/m}2^{-M}
+2^{-Mm}).
\end{equation}
Recalling the lower bound ~\eqref{eq:Prelim}
\begin{equation*}
N(\Omega;\varPi)\ge \lambda(\varPi) p^m \sum_{i=1}^M \sum_{\Gamma \in \cB_i} \mu(\Gamma) + O(R),
\end{equation*}
where
$$R
= p^{m-1/2} (\log p)^{n+1}\sum_{i=1}^M  \# \cB_i 2^{-i(m-1)}.
$$
We use ~\eqref{eq:Card vws} to bound the term $R$. First we note if $\# \cB_i >0$ we must have $2^{-im}\leq \mu(\Omega)$. Hence
\begin{equation*}
\begin{split}
\sum_{i=1}^M  \# \cB_i 2^{-i(m-1)}&= \sum_{i\ge -\log{\mu(\Omega)}/(m\log{2})}^{M}\# \cB_i2^{-i(m-1)} \\
&\ll \sum_{i\ge -\log{\mu(\Omega)}/(m\log{2})}^{M} \(\mu(\Omega)^{(m-1)/m}+2^{-i(m-1)}\) \\
&\ll M\mu(\Omega)^{(m-1)/m}+\mu(\Omega)^{(m-1)/m}\\ &\ll M\mu(\Omega)^{(m-1)/m}
\end{split} 
\end{equation*}
so that
\begin{equation*}
\begin{split}
R\ll M\mu(\Omega)^{(m-1)/m}p^{m-1/2} (\log p)^{n+1}.
\end{split} 
\end{equation*}
By ~\eqref{eq:Aprox vws} we have,
\begin{equation*}
\begin{split}
\sum_{i=1}^M \sum_{\Gamma \in \cB_i} \mu(\Gamma)&= \mu\(\bigcup_{i=1}^M  \bigcup_{\Gamma \in \cB_i}\Gamma\) =  \mu\(\Omega\) + O(\mu(\Omega)^{(m-1)/m}2^{-M}+2^{-Mm}).
\end{split} 
\end{equation*}
Hence
\begin{equation*}
\begin{split}
N(\Omega;\varPi)\ge \lambda(\varPi) \mu(\Omega) p^m+O(\mu(\Omega&)^{(m-1)/m}p^{m}2^{-M}+p^{m}2^{-Mm}
\\ &   + M\mu(\Omega)^{(m-1)/m}p^{m-1/2} (\log p)^{n+1}).
\end{split} 
\end{equation*}

Since $D(\Omega) \le 1$, we can   assume that
$$
\mu(\Omega) \ge c_0 p^{-m/2} (\log p)^{m(n+2)}
$$ 
for a sufficiently large constant $c_0 > 0$. 
Thus,  choosing $M$ so that
$$2^{M}\leq p \leq 2^{(M+1)},$$
 gives
\begin{equation*}
\begin{split}
N(\Omega;\varPi)&\ge \lambda(\varPi) \mu(\Omega) p^{m}+O(\mu(\Omega)^{(m-1)/m}p^{m-1/2}(\log{p})^{n+2}).
\end{split} 
\end{equation*}
 The upper bounds for $N(\Omega;\varPi)$ and $D(\Omega)$ 
 follow the same method as in the proof of Theorem~\ref{thm:Omega}.

\section{Proof of Theorem~\ref{system}}
Given $\Omega$ very well shaped, we consider the same constructions in the proof of Theorem~\ref{thm:Omega}. As in Theorem~\ref{thm:Omega vws} we have the bound
\begin{equation}
\label{very well shaped cubes bound}
\# \cB_i \ll 1+\mu(\Omega)^{1-1/m}2^{i(m-1)}.
\end{equation}
The set inclusions
\begin{equation}
\label{omega}
\Omega  \setminus \Omega_\varepsilon^{-}  \subseteq \bigcup_{i=1}^M  \bigcup_{\Gamma \in \cB_i}\Gamma  \subseteq \Omega
\end{equation}
give the approximation
\begin{equation}
\label{very well shaped approx by cubes}
\mu \(\bigcup_{i=1}^{M}\bigcup_{\Gamma \in \cB_i}\Gamma \)=\mu(\Omega)+O \(\frac{\mu(\Omega)^{1-1/m}}{2^M}+\frac{1}{2^{mM}}\).
\end{equation}
Using Lemma~\ref{fouv} and~\eqref{omega},
$$
T_p(\Omega)\geq \sum_{i=1}^{M} \sum_{\Gamma\in \cB_i}T_p(\Gamma)=
\#\cX_p \sum_{n=1}^{M}\sum_{\Gamma \in \cB_i}\mu(\Gamma)+O(R),
$$
where
$$R=\sum_{i=1}^{M}\#B_i \( p^{(m-n)/2}(\log{p})^m+2^{-i(m-n-1)}p^{m-n-1/2}(\log{p})^{n+1}\).$$
By (\ref{very well shaped approx by cubes})
$$
\# \cX_p \sum_{n=1}^{M} \sum_{\Gamma \in \cB_i}\mu(\Gamma)=\# \cX_p \(\mu(\Omega)+
O \(\frac{\mu(\Omega)^{1-1/m}}{2^M}+\frac{1}{2^{mM}}\) \)$$
and using~\eqref{LaWe} we have
$$
\# \cX_p  \sum_{n=1}^{M} \sum_{\Gamma \in \cB_i}\mu(\Gamma)=\# \cX_p\mu(\Omega)
+O \(\frac{p^{m-n}\mu(\Omega)^{1-1/m}}{2^M}+\frac{p^{m-n}}{2^{mM}}\).
$$
For the term $R$, by~\eqref{very well shaped cubes bound}
\begin{equation*}
\begin{split}
R&\ll  \sum_{i=1}^{M}\(p^{(m-n)/2}(\log{p})^m+ (p2^{-i})^{m-n-1}p^{1/2}(\log{p})^{n+1}\) \\
   & \qquad  +\sum_{i=1}^{M}\(\mu(\Omega)^{1-1/m}2^{i(m-1)}p^{(m-n)/2}(\log{p})^m+(p2^{-i})^{m-n-1}p^{1/2}(\log{p})^{n+1}\)
\\ &\ll Mp^{(m-n)/2}(\log{p})^m+ p^{m-n-1/2}(\log{p})^{n+1} \sum_{i=1}^{M}2^{-i(m-n-1)} \\    &  \qquad  +\mu(\Omega)^{1-1/m}p^{(m-n)/2}\sum_{i=1}^{M}2^{i(m-1)} (\log{p})^m  \\ & \qquad  + 
\mu(\Omega)^{1-1/m}p^{m-n-1/2} \sum_{i=1}^{M}2^{in} (\log{p})^{n+1} \\
&\ll M \( p^{(m-n)/2}(\log{p})^m+p^{m-n-1/2}(\log{p})^{n+1}\) \\ & \qquad
+\mu(\Omega)^{1-1/m} \( 2^{M(m-1)}p^{(m-n)/2}(\log{p})^m+2^{Mn}p^{m-n-1/2}(\log{p})^{n+1}\).
\end{split} 
\end{equation*}
Hence we have
\begin{equation}
\label{upper bound T}
T_p(\Omega)\geq \# \cX_p \mu(\Omega)+O(R_1+R_2+R_3)
\end{equation}
with
\begin{equation*}
\begin{split}
R_1&=\frac{p^{m-n}\mu(\Omega)^{1-1/m}}{2^M}+\frac{p^{m-n}}{2^{mM}}, \\
R_2&=\mu(\Omega)^{1-1/m} \( 2^{M(m-1)}p^{(m-n)/2}(\log{p})^m+2^{Mn}p^{m-n-1/2}(\log{p})^{n+1}\), \\
R_3&=Mp^{m-n-1/2}(\log{p})^{n+1}.
\end{split} 
\end{equation*}
It is clear that for the bound to be nontrivial we have to choose $M  = O(\log p)$, 
under which condition we have 
$$
R_3=p^{m-n-1/2}(\log{p})^{n+2}.
$$
Now considering all four ways of balancing the terms of $R_1$ and $R_2$, after 
straight forward calculations we conclude that the optimal choice of $M$ is defined by 
the condition
$$2^{-M}\leq p^{-1/2(n+1)}\log{p}<2^{-M+1}.$$
that  balances the first term of $R_1$ and the second term of $R_2$.
This gives
\begin{equation}
\begin{split}
\label{eq:bound R}
R&\ll p^{m-n-1/2(n+1)}\mu(\Omega)^{1-1/m}\log{p}+p^{m-n-m/2(n+1)}(\log{p})^m \\
& \qquad\quad  +p^{(m-n)-1/2(n+1)-n(m-1-n)/2(n+1)}\mu(\Omega)^{1-1/m}\log{p}  \\
& \qquad\qquad\qquad\qquad\qquad\qquad\qquad\qquad\quad+ p^{m-n-1/2}(\log{p})^{n+2}  \\
&\ll p^{m-n-1/2(n+1)}\mu(\Omega)^{1-1/m}\log{p}+p^{m-n-1/2}(\log{p})^{n+2}.
\end{split} 
\end{equation}
Hence by~\eqref{LaWe},
\begin{equation}
\begin{split}
\label{almost}
T_p(&\Omega)\\
&\geq \# \cX_p \( \mu(\Omega)+O \(\mu(\Omega)^{1-1/m}p^{-1/2(n+1)}\log{p}+p^{-1/2}(\log{p})^{n+2}\)\).
\end{split} 
\end{equation}
Although since 
$$(\Tm\setminus \Omega)_{\varepsilon}^{-}=\Omega_{\varepsilon}^{+}\leq C \(\mu(\Omega)^{1-1/m}\varepsilon+\varepsilon^{m}\)$$
we may repeat the above argument to get,
\begin{equation}
\begin{split}
\label{almost finished}
T_p(\Tm\setminus \Omega)&\geq \#\cX_p \mu(\Tm\setminus \Omega) \\ &  +O\(\# \cX_p\(\mu(\Omega)^{1-1/m}p^{-1/2(n+1)}\log{p}+p^{-1/2}(\log{p})^{n+2}\)\).
\end{split} 
\end{equation}
Finally, combining~\eqref{LaWe}, \eqref{almost} and~\eqref{almost finished},  gives
the desired result.

\end{document}